\documentclass{amsart}
\usepackage{amscd,amsmath,amssymb,amsthm,amsfonts,epsfig,graphics}

\theoremstyle{theorem}
\newtheorem{theorem}{Theorem}
\theoremstyle{definition}
\newtheorem{definition}{Definition}
\newtheorem{prop}{Proposition}
\newtheorem{remark}{Remark}
\newtheorem{lemma}{Lemma}
\newtheorem{example}{Example}

\title{Warped product submanifolds in metallic Riemannian manifolds}
\author{Cristina E. Hretcanu and Adara M. Blaga}


\begin{document}

\normalfont
 \begin{abstract}

     In this paper, we study the existence of proper warped product submanifolds in metallic (or Golden) Riemannian manifolds and we discuss about semi-invariant, semi-slant and, respectively, hemi-slant warped product submanifolds in metallic and Golden Riemannian manifolds. Also, we provide some examples of warped product submanifolds in Euclidean spaces.
     \end{abstract}
     \maketitle

\section{Introduction}

Warped products can be seen as a natural generalization of cartesian products. This concept appeared in mathematics starting with the J. F. Nash's studies, who proved an embedding theorem which states that every Riemannian manifold can be isometrically embedded into some Euclidean space. Also, Nash's theorem shows that every warped product $M_{1} \times_{f} M_{2}$ can be embedded as a Riemannian submanifold in some Euclidean space (\cite{Nash}, 1956).

Then, the study of warped product manifolds was continued by R. L. Bishop and B. O'Neill in (\cite{Bishop}, 1964), where they obtained fundamental properties of warped product manifolds and constructed a class of complete manifolds of negative curvature.

B. Y. Chen studied CR-submanifolds of a K\"{a}hler manifold which are warped products of holomorphic and totally real submanifolds, respectively (\cite{Chen3},\cite{Chen1},\cite{Chen2}). Also, in his new book (\cite{ChenBook}, 2017), he presents a multitude of properties for warped product manifolds and submanifolds, such as: warped products of Riemannian and K\"{a}hler manifolds, warped product submanifolds of K\"{a}hler manifolds (with the particular cases: warped product CR-submanifolds, warped product hemi-slant or semi-slant submanifolds of K\"{a}hler manifolds), CR-warped products in complex space forms and so on.

As a generalization of contact CR-submanifolds, slant and semi-slant submanifolds, J. L. Cabrerizo et al. introduced the notion of bi-slant submanifolds of almost contact metric manifolds (\cite{Cabrerizo2}).

B. Sahin studied the properties of warped product submanifolds of K\"{a}hler manifolds with a slant factor and he proved that the warped product semi-slant submanifolds $M_{T} \times_{f} M_{\theta}$ and $M_{\theta}\times_{f} M_{T}$ in K\"{a}hler manifolds are simply Riemannian products of $M_{T}$ and $M_{\theta}$, where $M_{\theta}$ is a proper slant submanifold of the underlying K\"{a}hler manifold (\cite{Sahin1}). Also, he studied warped product hemi-slant submanifolds of a K\"{a}hler manifold and he found some properties of warped product submanifolds of the form $M_{\perp} \times_{f} M_{\theta}$ (\cite{Sahin2}).

Semi-invariant submanifolds in locally product Riemannian manifolds were studied in (\cite{Sahin0},\cite{Atceken1}). Semi-slant submanifolds in locally Riemannian product manifolds were studied by M. At\c{c}eken which found that, in a locally Riemannian product manifold does not exist any warped product semi-slant submanifold of the form $M_{T} \times_{f} M_{\theta}$ nor of the form $M_{\perp} \times_{f} M_{\theta}$ such that $M_{T}$ is an invariant submanifold, $M_{\theta}$ is a proper slant submanifold and $M_{\perp}$ is an anti-invariant submanifold, but he found some examples of warped product semi-slant submanifolds of the form $M_{\theta} \times_{f} M_{T}$ and of the form $M_{\theta} \times_{f}  M_{\perp}$ (\cite{Atceken0}). Warped product submanifolds of the form $M_{\theta}\times_{f} M_{T}$ and $M_{\theta}\times_{f} M_{\perp}$ in Riemannian product manifolds were also studied by F. R. Al-Solamy and M. A. Khan (\cite{Solemy}).

Warped product pseudo-slant (named also hemi-slant) submanifolds of the form $M_{\theta} \times_{f} M_{\perp}$, where $M_{\theta}$ and $M_{\perp}$ are proper slant and, respectively, anti-invariant submanifolds, in a locally product Riemannian manifold were studied by S. Uddin et al. in (\cite{Uddin}).

Recently, warped product bi-slant submanifolds in K\"{a}hler manifolds were studied by S. Uddin et al. and some examples of this type of submanifolds in complex Euclidean spaces were constructed (\cite{Uddin2}). Moreover, L. S. Alqahtani et al. have shown that there is no proper warped product bi-slant submanifold other than pseudo-slant warped product in cosymplectic manifolds (\cite{Alqahtani}).

The authors of the present paper studied some properties of invariant, anti-invariant and slant submanifolds (\cite{Blaga_Hr}), semi-slant submanifolds (\cite{Hr6}) and, respectively, hemi-slant submanifolds (\cite{Hr7}) in metallic and Golden Riemannian manifolds and they obtained integrability conditions for the distributions involved in these types of submanifolds. Moreover, properties of metallic and Golden warped product Riemannian manifolds were presented in the two previews works of the authors (\cite{Blaga1},\cite{Blaga2}).

In the present paper, we study the existence of proper warped product bi-slant submanifolds in locally metallic Riemannian manifolds.
In Sections 2 and 3, we remind the main properties of metallic and Golden Riemannian manifolds and of their submanifolds. In Section 4, we discuss about slant and bi-slant submanifolds (with their particular cases: semi-slant and hemi-slant submanifolds) in locally metallic (or Golden) Riemannian manifolds.
In Section 5, we find some properties of warped product bi-slant submanifolds in metallic (or Golden) Riemannian manifolds and, in particular, we discuss about warped product semi-invariant, semi-slant and, respectively, hemi-slant submanifolds in locally metallic (or locally Golden) Riemannian manifolds. Moreover, we construct suitable examples.

\section{Preliminaries}

The metallic structure is a particular case of polynomial structure on a manifold, which was generally defined in (\cite{Goldberg1},\cite{Goldberg2}). The name of metallic number is given to the positive solution of the equation $x^{2}-px-q=0$ (where $p$ and $q$ are positive integer values), which is $\sigma _{p,q}=\frac{p+\sqrt{p^{2}+4q}}{2}$ (\cite{Spinadel}).

Metallic Riemannian manifolds and their submanifolds were defined and studied by C. E. Hretcanu, M. Crasmareanu and A. M. Blaga in (\cite{Hr4},\cite{Hr5}), as a generalization of the Golden Riemannian manifolds studied in (\cite{CrHr},\cite{Hr3},\cite{Hr2}).

Let $\overline{M}$ be an $m$-dimensional manifold endowed with a tensor field $J$ of type $(1,1)$ such that:
\begin{equation}\label{e1}
J^{2}= pJ+qI,
\end{equation}
for $p$, $q\in\mathbb{N}^*$, where $I$ is the identity operator on $\Gamma(T\overline{M})$. Then the structure $J$ is called a \textit{metallic structure}. If the Riemannian metric $\overline{g}$ is $J$-compatible, i.e.:
\begin{equation} \label{e2}
\overline{g}(JX, Y)= \overline{g}(X, JY),
\end{equation}
for any $X, Y \in \Gamma(T\overline{M})$, then $(\overline{M},\overline{g},J)$ is a {\it metallic Riemannian manifold} (\cite{Hr4}).

In particular, if $p=q=1$ one obtains the \textit{Golden structure} determined by a $(1,1)$-tensor field $J$ which verifies $J^{2}= J + I$. In this case, $(\overline{M},J)$ is called {\it Golden manifold}. If ($\overline{M}, \overline{g})$ is a Riemannian manifold endowed with a Golden structure $J$ such that the Riemannian metric $\overline{g}$ is $J$-compatible, then $(\overline{M},\overline{g},J)$ is a {\it Golden Riemannian manifold} (\cite{CrHr}).

From (\ref{e1}) and (\ref{e2}) we remark that the metric verifies:
\begin{equation} \label{e3}
\overline{g}(JX, JY)=\overline{g}(J^{2}X, Y) =p \overline{g}(JX,Y)+q \overline{g}(X,Y),
\end{equation}
for any $X, Y \in \Gamma(T\overline{M})$.
\normalfont

Any almost product structure $F$ on $\overline{M}$ induces two metallic structures on $\overline{M}$:
\begin{equation}\label{e4}
J= \frac{p}{2}I \pm \frac{2\sigma _{p, q}-p}{2}F,
\end{equation}
where $I$ is the identity operator on $\Gamma(T\overline{M})$. Also, on a metallic manifold $(\overline{M},J)$, there exist two complementary distributions $\mathcal{D}_{l}$ and $\mathcal{D}_{m}$
corresponding to the projection operators $l$ and $m$ given by (\cite{Hr4}):
\begin{equation}\label{e5}
l=-\frac{1}{2\sigma _{p, q}-p} J+\frac{\sigma _{p, q}}{2\sigma _{p, q}-p} I, \quad m=\frac{1}{2\sigma _{p, q}-p} J+\frac{\sigma _{p, q}-p}{2\sigma _{p, q}-p} I.
\end{equation}

\section{Submanifolds of metallic Riemannian manifolds}

We recall the basic properties of a metallic Riemannian structure and prove some immediate consequences of the Gauss and Weingarten
equations for an isometrically immersed submanifold in a metallic Riemannian manifold (\cite{Hr7},\cite{Hr6},\cite{Hr5}).

Let $M$ be an isometrically immersed submanifold in the metallic (or Golden) Riemannian manifold ($\overline{M}, \overline{g},J)$. Let $T_{x}M$ be the tangent space of $M$ in a point $x \in M$ and $T_{x}^{\bot }M$ the normal space of $M$ in $x$. The tangent space $T_x\overline{M}$ can be decomposed into the direct sum
$T_x\overline{M}=T_x M\oplus T_x^{\perp}M,$ for any $x\in M$. Let $i_{*}$ be the differential of the immersion $i: M \rightarrow\overline{M}$. Then the induced Riemannian metric $g$ on $M$ is given by $g(X, Y)=\overline{g}(i_{*}X, i_{*}Y)$, for any $X, Y \in \Gamma(TM)$. For the simplification of the notations, in the rest of the paper we shall denote by $X$ the vector field $i_{*}X$, for any $X \in \Gamma(TM)$.

Let $TX:=(J X)^T$ and $NX:=(J X)^{\perp}$, respectively, be the tangential and normal components of $JX$, for any $X \in \Gamma(TM)$ and $tV:=(J V)^{T}$, $nV:=(J V)^{\perp}$ be the tangential and normal components of $JV$, for any $V \in \Gamma(T^{\perp}M)$. Then we get:
\begin{equation}\label{e6}
(i)\: JX = TX + NX,  \quad (ii)\: JV = tV + nV,
\end{equation}
for any $X \in \Gamma(TM)$, $V \in \Gamma(T^{\perp}M)$.

The maps $T$ and $n$ are $\overline{g}$-symmetric (\cite{Blaga_Hr}):
\begin{equation}\label{e7}
(i)\: \overline{g}(TX,Y)=\overline{g}(X,TY), \quad (ii)\: \overline{g}(nU,V)=\overline{g}(U,nV),
\end{equation}
and
\begin{equation}\label{e8}
 \overline{g}(NX,V)=\overline{g}(X,tV),
\end{equation}
for any $X,Y \in \Gamma(TM)$ and $U,V\in \Gamma(T^{\perp}M)$.

If $M$ is a submanifold in a metallic Riemannian manifold $(\overline{M}, \overline{g}, J)$, then (\cite{Hr7}):
\begin{equation} \label{e9}
(i) \: T^{2}X = pTX+qX-tNX, \quad (ii) \: pNX= NTX+nNX ,
\end{equation}
\begin{equation} \label{e10}
(i) \: n^{2}V =  pnV+qV-NtV, \quad (ii) \: ptV= TtV+tnV,
\end{equation}
for any $X \in \Gamma(TM)$ and $V \in \Gamma(T^{\bot}M)$.

If $p=q=1$ and $M$ is a submanifold in a Golden Riemannian manifold $(\overline{M}, \overline{g}, J)$, then, for any $X \in \Gamma(TM)$ we get $ T^{2}X = TX+X-tNX$,  $NX= NTX+ nNX$ and for any $V \in \Gamma(T^{\bot}M)$ we get $n^{2}V = nV+V-NtV$,  $tV= TtV+tnV$.

Let $\overline{\nabla}$ and $\nabla $ be the Levi-Civita connections on $(\overline{M},\overline{g})$ and on its submanifold $(M,g)$, respectively. The Gauss and Weingarten formulas are given by:
\begin{equation}\label{e11}
(i) \: \overline{\nabla}_{X}Y=\nabla_{X}Y+h(X,Y), \quad  (ii) \: \overline{\nabla}_{X}V=-A_{V}X+\nabla_{X}^{\bot}V,
\end{equation}
for any $X, Y \in \Gamma(TM)$ and $V \in \Gamma(T^{\bot}M)$, where $h$ is the second fundamental form and $A_{V}$ is the shape operator, which are related by:
 \begin{equation}\label{e12}
 \overline{g}(h(X, Y),V)=\overline{g}(A_{V}X, Y).
 \end{equation}

If $(\overline{M},\overline{g}, J)$ is a metallic (or Golden) Riemannian manifold and $J$ is parallel with respect to the Levi-Civita connection $\overline{\nabla}$ on $\overline{M}$ (i.e. $\overline{\nabla}J=0$), then $(\overline{M},\overline{g}, J)$ is called a {\it locally metallic (or locally Golden) Riemannian manifold} (\cite{Hr5}).

The covariant derivatives of the tangential and normal components of $JX$ (and $JV$), $T$ and $N$ ($t$ and $n$), respectively, are given by:
\begin{equation}\label{e13}
(i) \: (\nabla_{X}T)Y=\nabla_{X}TY - T(\nabla_{X}Y), \quad (ii) \:(\overline{\nabla}_{X}N)Y=\nabla_{X}^{\bot}NY - N(\nabla_{X}Y),
\end{equation}
\begin{equation}\label{e14}
(i) \: (\nabla_{X}t)V=\nabla_{X}tV - t(\nabla_{X}^{\bot}V), \quad (ii) \:(\overline{\nabla}_{X}n)V=\nabla_{X}^{\bot}nV - n(\nabla_{X}^{\bot}V),
\end{equation}
for any $X$, $Y \in \Gamma(TM)$ and $V \in \Gamma(T^{\bot}M) $. From $(\ref{e1})$ it follows:
\begin{equation} \label{e15}
\overline{g}((\overline{\nabla}_XJ)Y,Z)=\overline{g}(Y,(\overline{\nabla}_XJ)Z),
\end{equation}
for any $X$, $Y$, $Z\in \Gamma(T\overline{M})$. Moreover, if $M$ is an isometrically immersed submanifold in the metallic Riemannian manifold $(\overline{M},\overline{g},J)$, then (\cite{Blaga3}):
\begin{equation}\label{e16}
\overline{g}((\nabla_X T)Y,Z)=\overline{g}(Y,(\nabla_X T)Z),
\end{equation}
for any $X$, $Y$, $Z\in \Gamma(TM)$.

\begin{prop}
If $M$ is a submanifold in a locally metallic (or locally Golden) Riemannian manifold $(\overline{M},\overline{g},J)$, then the covariant derivatives of $T$ and $N$ verify:
\begin{equation}\label{e17}
(i)  (\nabla_{X}T)Y=A_{NY}X+th(X,Y), \quad (ii) \: (\overline{\nabla}_{X}N)Y=nh(X,Y)-h(X,TY),
\end{equation}
\begin{equation}\label{e18}
(i)  (\nabla_{X}t)V=A_{nV}X - TA_{V}X, \quad (ii) \: (\overline{\nabla}_{X}n)V=-h(X,tV)-NA_{V}X
\end{equation}
and
\begin{equation}\label{e19}
\overline{g}((\overline{\nabla}_{X}N)Y,V )= \overline{g}((\nabla_{X}t)V,Y),
\end{equation}
for any $X$, $Y \in \Gamma(TM)$ and $V \in \Gamma(T^{\bot}M)$ (\cite{Hr7}).
\end{prop}

\section{Bi-slant submanifolds in metallic or Golden Riemannian manifolds}

\begin{definition} (\cite{Blaga_Hr})
A submanifold $M$ in a metallic (or Golden) Riemannian manifold ($\overline{M}, \overline{g}, J)$ is called \textit{slant submanifold} if the angle $\theta(X_x)$ between $JX_x$ and $T_xM$ is constant, for any $x\in M$ and $X_x\in T_xM$. In such a case, $\theta=:\theta(X_x)$ is called the \textit{slant angle} of $M$ in $\overline{M}$ and it verifies:
\begin{equation}\label{e20}
\cos\theta =\frac{\overline{g}(JX,TX)}{\| JX \| \cdot \| TX \|}=\frac{\| TX \|}{\| JX \|}.
\end{equation}
 The immersion $i: M \rightarrow \overline{M} $ is named {\it slant immersion} of $M$ in $\overline{M}$.
\end{definition}
\normalfont

\begin{remark}
The invariant and anti-invariant submanifolds in metallic (or Golden) Riemannian manifolds ($\overline{M}, \overline{g}, J)$ are particular cases of slant submanifolds with the slant angle $\theta=0$ and $\theta=\frac{\pi}{2}$, respectively. A slant submanifold $M$ in $\overline{M}$, which is neither invariant nor anti-invariant, is called \textit{proper slant submanifold} and the immersion $i: M \rightarrow \overline{M} $ is called \textit{proper slant immersion} (\cite{Blaga_Hr}).
\end{remark}

\begin{definition}
Let $M$ be an immersed submanifold in a metallic (or Golden) Riemannian manifold $(\overline{M},\overline{g},J)$. A differentiable distribution $D$ on $M$ is called {\it slant distribution} if the angle $\theta_{D}$ between $JX_{x}$ and the vector subspace $D_{x}$ is constant, for any $x \in M$  and any nonzero vector field $X_{x} \in \Gamma(D_{x})$. The constant angle $\theta_{D}$ is called the {\it slant angle} of the distribution $D$ (\cite{Hr6}).
\end{definition}

\begin{definition}\label{d1}
A submanifold $M$ in a metallic (or Golden) Riemannian manifold $(\overline{M},\overline{g},J)$ is called {\it bi-slant} if there exist two orthogonal differentiable distributions $D_{1}$ and $D_{2}$ on $M$ such that:
\begin{enumerate}
\item $TM = D_{1}\oplus D_{2}$;
\item $J D_{1} \perp D_{2}$ and $J D_{2} \perp D_{1}$;
\item the distributions $D_{1}$ and $D_{2}$ are slant distribution with slant angles $\theta_{1}$ and $\theta_{2}$, respectively.
\end{enumerate}
 Moreover, $M$ is called {\it proper bi-slant submanifold} of $\overline{M}$ if $dim(D_{1})\cdot dim(D_{2}) \neq 0 $ and $\theta_{1}, \theta_{2} \in (0, \frac{\pi}{2})$.
\end{definition}

\begin{definition} (\cite{Hr7})
An immersed submanifold $M$ in a metallic (or Golden) Riemannian manifold $(\overline{M},\overline{g},J)$ is \textit{semi-slant} submanifold (\textit{hemi-slant} submanifold, respectively) if there exist two orthogonal distributions $D_{1}$ and $D_{2}$ on $M$ such that:
\begin{enumerate}
\item  $TM$ admits the orthogonal direct decomposition $TM=D_{1}\oplus D_{2}$;
\item  the distribution $D_{1}$ is an invariant distribution, i.e. $J(D_{1})=D_{1}$ ($D_{1}$ is an anti-invariant distribution, i.e. $J(D_{1}) \subseteq \Gamma(T^{\perp}M)$, respectively);
\item  the distribution $D_{2}$ is slant with the angle $\theta \in [0, \frac{\pi}{2}]$.
\end{enumerate}
If $dim(D_{1})\cdot dim(D_{2}) \neq 0$ and $\theta \in (0, \frac{\pi}{2})$, then $M$ is a proper semi-slant submanifold (hemi-slant submanifold, respectively) in the metallic (or Golden) Riemannian manifold $(\overline{M},\overline{g},J)$.
\end{definition}

Now we provide an example of bi-slant submanifold in a metallic and Golden Riemannian manifold $(\mathbb{R}^{4}, \langle\cdot, \cdot \rangle,J)$.

\begin{example}
Let $\mathbb{R}^{4}$ be the Euclidean space endowed with the usual Euclidean metric $\langle\cdot,\cdot\rangle$
and the immersion $i: M \rightarrow \mathbb{R}^{4}$, given by:
$$i(f_1,f_2):=\left(f_1 \cos t, \frac{\sigma}{\sqrt{q}} f_1 \sin t, f_2, f_2\right),$$
where $M :=\{(f_1,f_2) \mid  f_1, f_2 >0\}$, $ t \in [0, \frac{\pi}{2}]$ and $\sigma:=\sigma_{p,q}=\frac{p+\sqrt{p^{2}+4q}}{2}$ is the metallic number ($p, q \in N^{*}$).

We can find a local orthonormal frame on $TM$ given by
$Z_{1}=\cos t \frac{\partial}{\partial x_{1}} + \frac{\sigma}{\sqrt{q}} \sin t \frac{\partial}{\partial x_{2}}$
and $Z_{2}=\frac{\partial}{\partial x_{3}}+ \frac{\partial}{\partial x_{4}}$.
We define the metallic structure $J : \mathbb{R}^{4} \rightarrow \mathbb{R}^{4} $ by
$
 J(X_{1},X_{2},X_{3},X_{4}):=(\sigma X_{1}, \overline{\sigma} X_{2}, \sigma X_{3}, \overline{\sigma} X_{4} ),
 $
which verifies $J^{2}=p J + q I$ and $\langle JX, Y\rangle = \langle X, JY\rangle$, for any $X,Y \in \mathbb{R}^{4}$, where $\overline{\sigma}:=p-\sigma$.

Since
$J Z_{1}=\sigma \cos t \frac{\partial}{\partial x_{1}} -\sqrt{q} \sin t \frac{\partial}{\partial x_{2}}$
and $JZ_{2}= \sigma \frac{\partial}{\partial x_{3}}+ \overline{\sigma} \frac{\partial}{\partial x_{4}}$, we remark that
$\langle JZ_{1}, Z_{1}\rangle = \sigma \cos 2t $ and $\langle JZ_{2}, Z_{2}\rangle = \sigma + \overline{\sigma}$.

Moreover,
$\|Z_{1}\|^{2}=\frac{p \sigma \sin^{2} t + q}{q}$, $\|J Z_{1}\|^{2}=p \sigma \cos^{2} t + q$, $\|Z_{2}\|^{2}=2$, $\|J Z_{2}\|^{2}=\sigma^{2} + \overline{\sigma}^{2}$.

Thus, $\cos \theta_{1}=\frac{2\sqrt{q} \cos 2t}{\sqrt{p^{2}\sin^{2} 2t +4q}}$ and
$\cos \theta_{2}=\frac{\sigma + \overline{\sigma}}{\sqrt{2(\sigma^{2} + \overline{\sigma}^{2})}}$.

Let us consider $D_{1}:= span \{Z_{1}\}$ and $D_{2}:= span \{Z_{2}\}$. Then, the distributions $D_{1}$ and $D_{2}$ satisfy the conditions from Definition \ref{d1}. Consequently, the submanifold $M_{\theta_{1},\theta_{2}}$ with $T M_{\theta_{1},\theta_{2}}= D_{1} \oplus D_{2}$, with the metric $g := (\frac{p \sigma}{q} \sin^{2}t +1) g_{1} + 2 g_{2}$, is a bi-slant submanifold in the metallic Riemannian manifold $(\mathbb{R}^{4}, \langle\cdot, \cdot \rangle,J)$, where $g_{1}$ and $g_{2}$ are the metric tensors of the integral manifolds $M_{1}$ and $M_{2}$ of the distributions $D_{1}$ and $D_{2}$.

Moreover, for $t=0$ ($t=\frac{\pi}{4}$) we obtain $\theta_{1}=0$ ($\theta_{1}=\frac{\pi}{2}$, respectively). Then, the submanifold $M_{0,\theta_{2}}$ ($M_{\frac{\pi}{2},\theta_{2}}$) is a semi-slant (a hemi-slant, respectively) submanifold in the metallic Riemannian manifold $(\mathbb{R}^{4}, \langle\cdot, \cdot \rangle,J)$.

In particular, for $p=q=1$, the immersion $i: M \rightarrow \mathbb{R}^{4}$ is given by
$i(f_1,f_2):=(f_1\cos t, \phi f_1 \sin t, f_2, f_2)$ and the Golden structure has the form
$
 J(X_{1},X_{2},X_{3},X_{4}):=(\phi X_{1}, \overline{\phi} X_{2}, \phi X_{3}, \overline{\phi} X_{4} ),
 $
 where $\phi:=\sigma_{1,1}=\frac{1+\sqrt{5}}{2}$ is the Golden number and $\overline{\phi}:=1-\phi$.
 Since
$J Z_{1}=\phi \cos t \frac{\partial}{\partial x_{1}} - \sin t \frac{\partial}{\partial x_{2}}$
and $J Z_{2}= \phi \frac{\partial}{\partial x_{3}}+ \overline{\phi} \frac{\partial}{\partial x_{4}}$, we remark that $\cos \theta_{1}=\frac{2\cos 2t}{\sqrt{\sin^{2} 2t +4}}$ and
$\cos \theta_{2}=\frac{\phi + \overline{\phi}}{\sqrt{2(\phi^{2} + \overline{\phi}^{2})}}$.

The distributions $D_{1}:= span \{Z_{1}\}$ and $D_{2}:= span \{Z_{2}\}$ satisfy the conditions from Definition \ref{d1}. Thus, the submanifold $M_{\theta_{1},\theta_{2}}$ with $T M_{\theta_{1},\theta_{2}}= D_{1} \oplus D_{2}$ and the metric $g := ( \phi \sin^{2}t +1) g_{1} + 2 g_{2}$ is a bi-slant submanifold in the metallic Riemannian manifold $(\mathbb{R}^{4}, \langle\cdot, \cdot \rangle,J)$, where $g_{1}$ and $g_{2}$ are the metric tensors of the integral manifolds $M_{1}$ and $M_{2}$ of the distributions $D_{1}$ and $D_{2}$.
Moreover, for $t=0$ ($t=\frac{\pi}{4}$) we obtain $\theta_{1}=0$ ($\theta_{1}=\frac{\pi}{2}$, respectively). Then, the submanifold $M_{0,\theta_{2}}$ ($M_{\frac{\pi}{2},\theta_{2}}$) is a semi-slant (a hemi-slant, respectively) submanifold in the Golden Riemannian manifold $(\mathbb{R}^{4}, \langle\cdot, \cdot \rangle,J)$.
\end{example}

\begin{remark}
If $M$ is a bi-slant submanifold in a metallic Riemannian manifold $(\overline{M},\overline{g},J)$ with the orthogonal distribution $D_{1}$ and $D_{2}$ and the slant angles $\theta_{1}$ and $\theta_{2}$, respectively, then
$$
JX=P_{1}TX+P_{2}TX+NX = TP_{1}X+ TP_{2} X + N P_{1}X + NP_{2}X,
$$
for any $X \in \Gamma(TM)$, where $P_1$ and $P_2$ are the projection operators on $\Gamma(D_1)$ and $\Gamma(D_2)$, respectively.
\end{remark}

\begin{prop} (\cite{Hr6})
If $M$ is a bi-slant submanifold in a metallic (or Golden) Riemannian manifold $(\overline{M},\overline{g},J)$, with the slant angles $\theta_{1}=\theta_{2}=\theta$ and $g(JX,Y)=0$, for any $X \in \Gamma(D_{1})$, $Y \in \Gamma(D_{2})$, then $M$ is a slant submanifold in the metallic Riemannian manifold $(\overline{M},\overline{g},J)$ with the slant angle $\theta$.
\end{prop}

\begin{remark}
If $M$ is a bi-slant submanifold in a metallic (or Golden) Riemannian manifold $(\overline{M},\overline{g},J)$ such that $TM = D_{1}\oplus D_{2}$, $dim(D_{1})\cdot dim(D_{2}) \neq 0$, where $D_{2}$ is the slant distribution (with the slant angle $\theta$), then we get:
\begin{enumerate}
\item $M$ is an invariant submanifold if $\theta=0$ and $D_{1}$ is invariant;
\item $M$ is an anti-invariant submanifold if $\theta = \frac{\pi}{2}$ and $D_{1}$ is anti-invariant;
\item $M$ is a proper semi-invariant submanifold if $D_{1}$ is invariant and $D_{2}$ is anti-invariant. The semi-invariant submanifold is a particular case of semi-slant submanifold (hemi-slant submanifold), with the slant angle $\theta = \frac{\pi}{2}$ ($\theta =0$, respectively).
\end{enumerate}
\end{remark}

\begin{prop} (\cite{Hr6})
Let $M$ be an isometrically immersed submanifold in the metallic Riemannian manifold $(\overline{M}, \overline{g}, J)$. If $M$ is a slant submanifold with the slant angle $\theta$, then:
\begin{equation}\label{e21}
\overline{g}(TX,TY)=\cos^2\theta[p\overline{g}(X,TY)+q \overline{g}(X,Y)]
\end{equation}
\begin{equation}\label{e22}
\overline{g}(NX,NY)=\sin^2\theta[p\overline{g}(X,TY)+q\overline{g}(X,Y)],
\end{equation}
for any $X$, $Y\in \Gamma(TM)$ and
\begin{equation}\label{e23}
(i) \: T^2=\cos^2\theta(pT+qI), \quad (ii) \nabla(T^2)=p\cos^2\theta(\nabla T).
\end{equation}
where $I$ is the identity on $\Gamma(TM)$.
\end{prop}

\begin{prop} (\cite{Hr7},\cite{Hr6})
If $M$ is a semi-slant submanifold (hemi-slant submanifold, respectively) in the metallic Riemannian manifold $(\overline{M},\overline{g},J)$ with the slant angle $\theta$ of the distribution $D_{2}$, then:
\begin{equation}\label{e24}
\overline{g}(TP_{2}X,TP_{2}Y)=\cos^2 \theta[p \overline{g}(TP_{2}X,P_{2}Y)+q \overline{g}(P_{2}X,P_{2}Y)]
\end{equation}
\begin{equation}\label{e25}
\overline{g}(NX,NY)=\sin^2 \theta[p \overline{g}(TP_{2}X,P_{2}Y)+q \overline{g}(P_{2}X,P_{2}Y)],
\end{equation}
for any $X$, $Y \in \Gamma(TM)$.
\end{prop}

\section{Semi-invariant, semi-slant and hemi-slant warped product submanifolds in metallic Riemannian manifolds}

In this section we present some results regarding the existence and nonexistence of semi-invariant, semi-slant and hemi-slant warped product submanifolds in metallic Riemannian manifolds $(\overline{M},\overline{g},J)$  and we give some examples of these types of submanifolds in metallic (or Golden) Riemannian manifolds.

In (\cite{Blaga1}), the authors of this paper introduced the Golden warped product Riemannian manifold and provided a necessary and sufficient condition for the warped product of two locally Golden Riemannian manifolds to be locally Golden. Moreover, the subject was continued in the paper (\cite{Blaga2}), where the authors characterized the metallic structure on the product of two metallic manifolds in terms of metallic maps and provided a necessary and sufficient condition for the warped product of two locally metallic Riemannian manifolds to be locally metallic.

Let $(M_1,g_1)$ and $({M_2},g_2)$ be two Riemannian manifolds of dimensions $n_{1}>0$ and $n_{2}>0$, respectively. We denote by $\pi_1$ and $\pi_2$ the projection maps from the product manifold ${M_1}\times {M_2}$ onto ${M_1}$ and ${M_2}$, respectively and by $\widetilde{\varphi}:=\varphi \circ \pi_1$ the lift to ${M_1}\times {M_2}$ of a smooth function $\varphi$ on ${M_1}$. ${M_1}$ is called \textit{the base} and ${M_2}$ is \textit{the fiber} of ${M_1}\times {M_2}$. The unique element $\widetilde{X}$ of $\Gamma(T({M_1}\times {M_2}))$ that is $\pi_1$-related to $X\in \Gamma(T{M_1})$ and to the zero vector field on ${M_2}$ will be called the \textit{horizontal lift of $X$} and the unique element $\widetilde{V}$ of $\Gamma(T({M_1}\times {M_2}))$ that is $\pi_2$-related to $V\in \Gamma(T{M_2})$ and to the zero vector field on ${M_1}$ will be called the \textit{vertical lift of $V$}.
We denote by $\mathcal{L}({M_1})$ the set of all horizontal lifts of vector fields on ${M_1}$ and by $\mathcal{L}({M_2})$ the set of all vertical lifts of vector fields on ${M_2}$.

For $f: M_1 \longrightarrow (0,\infty)$ a smooth function on ${M_1}$, we consider the Riemannian metric $g$ on $M:={M_1}\times {M_2}$:
\begin{equation}\label{e26}
g:=\pi_1^* g_1+(f \circ \pi_1)^2 \pi_2^*g_2.
\end{equation}

\begin{definition} (\cite{Bishop})
The product manifold of ${M_1}$ and ${M_2}$ together with the Riemannian metric $g$ defined by (\ref{e26}) is called \textit{the warped product} of ${M_1}$ and ${M_2}$ by the warping function $f$.
\end{definition}

In the next considerations we shall denote by $(f \circ \pi_1)^2=: f^2$, $\pi_1^* g_1=: g_1$ and $\pi_2^*g_2=:g_2$, respectively.

\begin{definition} (\cite{ChenBook})
A warped product manifold $M:={M_1}\times_f {M_2}$ is called \textit{trivial} if the warping function $f$ is constant. In this case, ${M_1}\times_f {M_2}$ is the Riemannian product ${M_1}\times {M_2}_f$, where ${M_2}_f$ is the manifold $M_2$ equipped with the metric $f^2 g_2$ (which is homothetic to $g_2$).
\end{definition}

The Levi-Civita connection $\nabla$ on $M:={M_1}\times_f {M_2}$ is related to the Levi-Civita connections on $M_1$ and $M_2$, as follows:

\begin{lemma} (\cite{Atceken0} and \cite{ChenBook}, pag. 49)\label{1}
For $X,Y \in \Gamma(T{M_1})$ and $Z,W \in \Gamma(T{M_2})$, we have on $M:={M_1}\times_f {M_2}$ that:
\begin{enumerate}
      \item $\nabla_{X}Y \in \mathcal{L}({M_1})$;
      \item $\nabla_{X}Z=\nabla_{Z}X = X (\ln f) Z$;
      \item $\nabla_{Z}W =\nabla_{Z}^{M_2}W -\frac{grad f}{f}g(Z,W)$,
 \end{enumerate}
where $\nabla$ and $\nabla^{M_2}$ denote the Levi-Civita connections on $M$ and $M_2$, respectively.
\end{lemma}

\begin{prop} (\cite{Bishop})
The warped product manifold $M:={M_1}\times_f {M_2}$ is characterized by the fact that $M_1$ is totally geodesic and
$M_2$ is a totally umbilical submanifold of $M$, respectively.
\end{prop}

\begin{definition}
A warped product $M_{1} \times_{f} M_{2}$ of two slant submanifolds $M_{1}$ and $M_{2}$ in a metallic Riemannian manifold $(\overline{M},\overline{g},J)$ is called a \textit{warped product bi-slant submanifold}.
A warped product bi-slant submanifold $M_{1} \times_{f} M_{2}$ is called \textit{proper} if both submanifolds  $M_{1}$ and $M_{2}$ are proper slant in $(\overline{M},\overline{g},J)$.
\end{definition}

\begin{definition}
If $M:={M_1}\times_f {M_2}$ is a warped product submanifold in a metallic Riemannian manifold $(\overline{M},\overline{g},J)$ such that one of the components $M_{i}$ ($i \in \{1,2\}$) is an invariant submanifold (respectively, anti-invariant submanifold) in $\overline{M}$ and the other one is a slant submanifold in $\overline{M}$, with the slant angle $\theta \in [0, \frac{\pi}{2}]$, then we call the submanifold $M$ \textit{warped product semi-slant} (respectively, \textit{hemi-slant}) \textit{submanifold} in the metallic Riemannian manifold $(\overline{M},\overline{g},J)$.
\end{definition}

\begin{remark}
In particular, if $M:={M_1}\times_f {M_2}$ is a warped product semi-slant (respectively, hemi-slant) submanifold in a metallic Riemannian manifold $(\overline{M},\overline{g},J)$ such that the slant angle $\theta = \frac{\pi}{2}$ (respectively, $\theta = 0$), then we obtain $M:={M_{T}}\times_f{M_{\perp}}$ or $M:={M_{\perp}}\times_f{M_{T}}$ and $M$ is a warped product semi-invariant submanifold in $\overline{M}$.
\end{remark}

\begin{lemma} \label{l2}
Let $M:={M_1}\times_f {M_2}$ be a warped product bi-slant submanifold in a locally metallic (or locally GOlden) Riemannian manifold $(\overline{M},\overline{g},J)$. Then, for any $X,Y \in \Gamma(TM_{1})$ and $Z,W \in \Gamma(TM_{2})$ we have:
\begin{equation}\label{e27}
    (i) \:  \overline{g}(h(X,Y),NZ)=-\overline{g}(h(X,Z),NY); \quad (ii)\: \overline{g}(h(X,Z),NW)=0;
\end{equation}
\begin{equation}\label{e28}
     \overline{g}(h(Z,W),NX)= TX(\ln f)\overline{g}(Z,W)- X(\ln f)\overline{g}(Z,TW).
\end{equation}
\end{lemma}
\begin{proof}
For any $X,Y \in \Gamma(TM_{1})$ and $Z,W \in \Gamma(TM_{2})$ we get
$$ \overline{g}(h(X,Y),NZ)= \overline{g}(\overline{\nabla}_{X}Y,JZ)-\overline{g}(\overline{\nabla}_{X}Y,TZ).$$
By using $(\overline{\nabla}_{X}J)Y=0$ and (\ref{e2}) we obtain
$$\overline{g}(\overline{\nabla}_{X}Y,JZ)=\overline{g}(\overline{\nabla}_{X}JY,Z)=\overline{g}(\overline{\nabla}_{X}TY,Z)+\overline{g}(\overline{\nabla}_{X}NY,Z).$$

Moreover, using $\overline{g}(\overline{\nabla}_{X}Y,TZ)=-\overline{g}(\overline{\nabla}_{X}TZ,Y)$ and (\ref{e11}), we obtain
$$\overline{g}(h(X,Y),NZ)=\overline{g}(\nabla_{X}TY,Z)-\overline{g}(A_{NY}X,Z)+\overline{g}(Y,\nabla_{X}TZ).$$
By using Lemma 1 $(2)$, we have
$$\overline{g}(\nabla_{X}TY,Z)=-\overline{g}(\nabla_{X}Z,TY)=-X(\ln f)\overline{g}(TY,Z)$$
and $\overline{g}(Y,\nabla_{X}TZ)=X(\ln f)\overline{g}(Y,TZ)$.

Thus, from (\ref{e7})(i) we get $\overline{g}(h(X,Y),NZ)=-\overline{g}(A_{NY}X,Z)$ which implies (\ref{e27})(i).

For any $X \in \Gamma(TM_{1})$ and $Z,W \in \Gamma(TM_{2})$ we get
$$ \overline{g}(h(X,Z),NW)= \overline{g}(\overline{\nabla}_{X}Z,JW)-\overline{g}(\overline{\nabla}_{X}Z,TW).$$

By using $(\overline{\nabla}_{X}J)Z=0$, (\ref{e2}), (\ref{e6})(i) and (\ref{e11})(ii) we obtain
$$ \overline{g}(h(X,Z),NW)= \overline{g}(\nabla_{X}TZ,W)-\overline{g}(A_{NZ}X,W)-\overline{g}(\nabla_{X}Z,TW)$$ and using $(2)$ from Lemma 1, we have
$$ \overline{g}(h(X,Z),NW)= X (\ln f)[\overline{g}(TZ,W)-\overline{g}(Z,TW)]-\overline{g}(h(X,W),NZ).$$

Thus, from (\ref{e7})(i) we get
$ \overline{g}(h(X,Z),NW)= -\overline{g}(h(X,W),NZ).$

On the other hand, we have
 $$ \overline{g}(h(X,Z),NW)=\overline{g}(\nabla_{Z}TX,W)-\overline{g}(A_{NX}Z,W)-\overline{g}(\nabla_{Z}X,TW)$$
and we obtain $$ \overline{g}(h(X,Z),NW)=TX (\ln f)\overline{g}(Z,W)-X (\ln f)\overline{g}(Z,TW)-\overline{g}(h(Z,W),NX).$$
 After interchanging $Z$ by $W$ and using (\ref{e7})(i), we have
 $$\overline{g}(h(X,Z),NW)=\overline{g}(h(X,W),NZ)=-\overline{g}(h(X,Z),NW),$$ which implies (\ref{e27})(ii).

For any $X \in \Gamma(TM_{1})$ and $Z,W \in \Gamma(TM_{2})$ we get
$$ \overline{g}(h(Z,W),NX)= \overline{g}(\overline{\nabla}_{Z}W,JX)-\overline{g}(\overline{\nabla}_{Z}W,TX).$$

By using $(\overline{\nabla}_{Z}J)W=0$, (\ref{e2}) and (\ref{e11}) we obtain
$$ \overline{g}(h(Z,W),NX)= \overline{g}(\nabla_{Z}TW,X)-\overline{g}(A_{NW}Z,X)-\overline{g}(\nabla_{Z}W,TX).$$

By using (\ref{e27})(ii), we have $\overline{g}(A_{NW}Z,X)=\overline{g}(h(Z,X),NW)=0$ and we get
$$ \overline{g}(h(Z,W),NX)= -\overline{g}(TW,\nabla_{Z}X)+\overline{g}(W,\nabla_{Z} TX)$$
and using Lemma 1 (2) we obtain (\ref{e28}).
\end{proof}

\begin{prop} \label{p1}
Let $M:={M_T}\times_f {M_\perp}$ be a semi-invariant submanifold in a locally metallic Riemannian manifold $(\overline{M},\overline{g},J)$ (i.e. $M_{T}$ is invariant and $M_{\perp}$ is an anti-invariant submanifold in $\overline{M}$). Then, we have:
\begin{equation}\label{e100}
    TX(\ln f) = -\frac{q}{p} X(\ln f),
\end{equation}
for any $X \in \Gamma(TM_{T})$. Moreover, if $M$ is a semi-invariant submanifold in a locally Golden Riemannian manifold, then we get $ TX(\ln f) = - X(\ln f)$.
\end{prop}
\begin{proof}
By using $(\overline{\nabla}_{Z}J)X=0$ and (\ref{e3}) we obtain
\begin{equation}\label{e101}
 q\overline{g}(\nabla_{Z}X,W)=q\overline{g}(\overline{\nabla}_{Z}X,W)=\overline{g}(J\overline{\nabla}_{Z}X,JW)-p\overline{g}(J\overline{\nabla}_{Z}X,W).
 \end{equation}
 Moreover, we get
$$q\overline{g}(\nabla_{Z}X,W)=\overline{g}(\overline{\nabla}_{Z}JX,JW)-p\overline{g}(\overline{\nabla}_{Z}JX,W),$$
for any $X \in \Gamma(TM_{T})$ (i.e. $JX=TX$) and $Z,W \in \Gamma(TM_{\perp})$ (i.e. $JW=NW$). Thus, we get
$$ q\overline{g}(\nabla_{Z}X,W)=\overline{g}(\overline{\nabla}_{Z}TX,NW)-p\overline{g}(\overline{\nabla}_{Z}TX,W)=\overline{g}(h(TX,Z),NW)-p\overline{g}(\nabla_{Z}TX,W).$$
By using (\ref{e27})(ii), we get $\overline{g}(h(TX,Z),NW)=0$ and from Lemma 1(2) we obtain
$$ q X(\ln f) \overline{g}(Z,W)=-p TX(\ln f) \overline{g}(Z,W).$$ Thus, for the non-null vector field $Z=W \in \Gamma(TM_{\perp})$, we obtain (\ref{e100}).
\end{proof}

Using a similar idea as in (\cite{Atceken0}, Theorem 3.1), we get a property of a warped product semi-invariant submanifold in a locally metallic (locally Golden) Riemannian manifold of the form ${M_{T}}\times_f {M_{\perp}}$.

\begin{theorem}
Let $M:={M_{T}}\times_f {M_{\perp}}$ be a warped product semi-invariant submanifold in a locally metallic (or locally Golden) Riemannian manifold $(\overline{M},\overline{g},J)$ (i.e. $M_{T}$ is invariant and $M_{\perp}$ is an anti-invariant submanifold in $\overline{M}$). Then $M:={M_{T}}\times_f {M_{\perp}}$ is a non proper warped product submanifold in $\overline{M}$ (i.e. the warping function $f$ is constant on the connected components of $M_{T}$).
\end{theorem}

\begin{proof}
We have $JX=TX$, for any $X \in \Gamma(TM_{T})$ and $JW=NW$, for any $Z,W \in \Gamma(TM_{\perp})$. By using (\ref{e3}) and $\overline{\nabla}J=0$, we get:
$$\overline{g}(\nabla_{Z}X,W)=\overline{g}(\overline{\nabla}_{Z}X,W)=
\frac{1}{q}\overline{g}(\overline{\nabla}_{Z}JX,JW)-\frac{p}{q}\overline{g}(\overline{\nabla}_{Z}X,JW),$$
where $\nabla$ and $\overline{\nabla}$ denote the Levi-Civita connections on $M$ and $\overline{M}$, respectively.
Thus, from (\ref{e27})(ii) we get
$$
q\overline{g}(\nabla_{Z}X,W)=\overline{g}(h(Z,TX),NW)-p \overline{g}(h(X,Z),NW)=0,
$$
for any $X \in \Gamma(TM_{T})$, $Z,W \in \Gamma(TM_{\perp})$.
By using Lemma 1(2) we have $q X (\ln f) \overline{g}(Z,W)=0$ (where $q \in N^{*}$). Thus, for any non-null $Z=W \in \Gamma(TM_{\perp})$, we have $X (\ln f) \| Z \|^{2}=0$ and it follows $X (\ln f)=0$, for any $X \in \Gamma(TM_{T})$, which implies that $f$ is a constant function on the connected components of $M_{T}$.
\end{proof}

Now we provide examples of warped product semi-slant submanifold of the type $M:=M_{\perp}\times_fM_{T}$  in a metallic (and Golden) Riemannian manifold $(\overline{M},\overline{g},J)$.

\begin{example}
Let $\mathbb{R}^{5}$ be the Euclidean space endowed with the usual Euclidean metric $\langle\cdot,\cdot\rangle$.
Let $i: M \rightarrow \mathbb{R}^{5}$ be the immersion given by:
$$i(f,\alpha,\beta):=\left(f \sin \alpha, f \cos \alpha, f \sin \beta, f \cos \beta, \sqrt{\frac{p \sigma}{q}}f\right),$$
where $M :=\{(f,\alpha,\beta) \mid  f>0, \alpha, \beta \in (0, \frac{\pi}{2})\}$ and $\sigma:=\sigma_{p,q}=\frac{p+\sqrt{p^{2}+4q}}{2}$ is the metallic number ($p, q \in N^{*}$).

We can find a local orthonormal frame on $TM$ given by:
 $$Z_{1}= \sin \alpha \frac{\partial}{\partial x_{1}} + \cos \alpha \frac{\partial}{\partial x_{2}}+ \sin \beta \frac{\partial}{\partial x_{3}} + \cos \beta  \frac{\partial}{\partial x_{4}}+ \sqrt{\frac{p \sigma}{q}}\frac{\partial}{\partial x_{5}},$$
 $$ Z_{2}=f \cos \alpha \frac{\partial}{\partial x_{1}} - f \sin \alpha \frac{\partial}{\partial x_{2}}, \quad
 Z_{3}=f \cos \beta \frac{\partial}{\partial x_{3}} - f \sin \beta \frac{\partial}{\partial x_{4}}.$$

We define the metallic structure $J : \mathbb{R}^{5} \rightarrow \mathbb{R}^{5} $ by:
$$
 J(X_{1},X_{2},X_{3},X_{4},X_{5}):=(\sigma X_{1},\sigma X_{2}, \overline{\sigma} X_{3},\overline{\sigma} X_{4}, \overline{\sigma} X_{5} ),
$$
which verifies $J^{2}X=p J + q I$ and $\langle JX, Y\rangle = \langle X, JY\rangle$, for any $X$, $Y\in \mathbb{R}^{5}$.
 Since
 $$JZ_{1}=\sigma \sin \alpha \frac{\partial}{\partial x_{1}} +\sigma \cos \alpha \frac{\partial}{\partial x_{2}}+\overline{\sigma} \sin \beta \frac{\partial}{\partial x_{3}} + \overline{\sigma}\cos \beta  \frac{\partial}{\partial x_{4}}+ \overline{\sigma}\sqrt{\frac{p \sigma}{q}}\frac{\partial}{\partial x_{5}},$$
   $$JZ_{2}= \sigma Z_{2}, \quad JZ_{3}= \overline{\sigma} Z_{3},$$
we remark that $JZ_{1} \perp span \{Z_{1},Z_{2},Z_{3}\} = TM$ (i.e. $\langle JZ_{1}, Z_{i}\rangle = 0$, for any $i\in \{1,2,3\}$), and $J Z_{2}, JZ_{3} \subseteq span \{Z_{2},Z_{3}\}$. We find that $\|Z_{1}\|^{2}=1+\frac{\sigma^{2}}{q}$ and $\|Z_{2}\|^{2}= \|Z_{3}\|^{2}=f^{2}$.

Let us consider $D_{1}= span \{Z_{1}\}$ and $D_{2}= span \{Z_{2}, Z_{3}\}$. The distributions $D_{1}$ and $D_{2}$ satisfy the conditions from Definition \ref{d1} and they are completely integrable. Moreover, $D_{1}$ is an anti-invariant distribution and $D_{2}$ is an invariant distribution with respect to $J$ .

Let $M_{\perp}$ and $M_{T}$ be the integral manifolds of $D_{1}$ and $D_{2}$, respectively.
Therefore, $M :=M_{\perp} \times_{f}M_{T}$ with the Riemannian metric tensor
$$g:=\left(1+\frac{\sigma^{2}}{q}\right) df^{2} + f^{2} (d \alpha^{2} + d\beta^{2})=g_{M_{\perp}}+ f^{2} g_{M_{T}}$$
is a warped product semi-invariant submanifold in the metallic Riemannian manifold $(\mathbb{R}^{5}, \langle\cdot,\cdot\rangle, J)$.

In particular, for $p=q=1$ and $\phi:=\sigma_{1,1}=\frac{1+\sqrt{5}}{2}$ the Golden number, the immersion $i: M \rightarrow \mathbb{R}^{5}$ is given by
$$i(f,\alpha,\beta):=(f \sin \alpha, f \cos \alpha, f \sin \beta, f \cos \beta, \sqrt{\phi} f),$$
and the Golden structure $J : \mathbb{R}^{5} \rightarrow \mathbb{R}^{5} $ is defined by
$$ J(X_{1},X_{2},X_{3},X_{4},X_{5}):=(\phi X_{1},\phi X_{2}, \overline{\phi} X_{3},\overline{\phi}X_{4}, \overline{\phi}X_{5}),$$
where $\overline{\phi}=1-\phi$.
If $M_{\perp}$ and $M_{T}$ are the integral manifolds of the distributions $D_{1}:=span\{Z_{1}\}$ and $D_{2}:=span\{Z_{2}, Z_{3}\}$, respectively and the metric on $M :=M_{\perp} \times_{f}M_{T}$  is given by
$$g:=(1+\phi^2) df^{2} + f^{2} (d \alpha^{2} + d\beta^{2})=g_{M_{\perp}}+ f^{2} g_{M_{T}},$$ then we obtain a warped product semi-invariant submanifold $(M,g)$ in the Golden Riemannian manifold $(\mathbb{R}^{5}, \langle\cdot,\cdot\rangle, J)$.
\end{example}

\begin{prop} \label{p2}
Let $M:={M_\perp}\times_f {M_T}$ be a warped product semi-invariant submanifold in a locally metallic (or locally Golden) Riemannian manifold $(\overline{M},\overline{g},J)$ (i.e. $M_{T}$ is invariant and $M_{\perp}$ is an anti-invariant submanifold in $\overline{M}$). Then, $f$ is constant on the connected components of $M_{1\perp}$ if and only if we have:
\begin{equation}\label{e102}
(T-pI)A_{JX}Z = - t\nabla_{Z}^{\perp}JX,
\end{equation}
for any $X \in \Gamma(TM_\perp)$ and $Z \in \Gamma(TM_T)$, where $I$ is the identity on $\Gamma(TM)$ and $\nabla^{\perp}$ is the normal connection on $\Gamma(T^{\perp}M)$.
\end{prop}
\begin{proof}
By using (\ref{e101}) and (\ref{e11}), we obtain
$$ q\overline{g}(\nabla_{Z}X,\overline{W})=\overline{g}(\overline{\nabla}_{Z}JX,J\overline{W})-p\overline{g}(\overline{\nabla}_{Z}JX,\overline{W}),$$
for any $X \in \Gamma(TM_{\perp})$ (i.e. $JX=NX$), $Z \in \Gamma(TM_{T})$ (i.e. $JZ=TZ$) and $\overline{W} \in \Gamma(TM)$. Thus, we get
$$ q\overline{g}(\nabla_{Z}X,\overline{W})=\overline{g}(J\overline{\nabla}_{Z}NX,\overline{W})-p\overline{g}(\overline{\nabla}_{Z}NX,\overline{W})$$
and from here we get
$$ q\overline{g}(\nabla_{Z}X,\overline{W})=-\overline{g}(T A_{NX}Z + t\nabla_{Z}^{\perp}NX,\overline{W})+p\overline{g}(A_{NX}Z,\overline{W}).$$
Thus, by using Lemma 1(2), we obtain
 $$  \overline{g}(q X(\ln f) Z,\overline{W})= -\overline{g}((T-pI)A_{NX}Z+t\nabla_{Z}^{\perp}NX,\overline{W}),$$
 for any $X \in \Gamma(TM_{\perp})$, $Z \in \Gamma(TM_{T})$ and $\overline{W} \in \Gamma(TM)$, which implies 
 $$q X(\ln f)Z= -(T-pI)A_{JX}Z - t\nabla_{Z}^{\perp}JX,$$. Thus $f$ is constant on the connected components of $M_{1\perp}$ if and only if 
(\ref{e102}) occurs.
\end{proof}

\begin{prop} \label{p3}
Let $M:={M_{1T}}\times_f {M_{2T}}$ be a warped product submanifold in a locally metallic (or locally Golden) Riemannian manifold $(\overline{M},\overline{g},J)$, where $M_{1T}$ and $M_{2T}$ are invariant submanifolds in $\overline{M}$. If $X(\ln f)\neq 0$, for any $X \in \Gamma(TM_{1T})$, then we have:
\begin{equation}\label{e103}
TX(\ln f)= \sigma X(\ln f) \quad or \quad TX(\ln f)= \overline{\sigma} X(\ln f),
\end{equation}
for any $X \in \Gamma(TM_{1T})$ and $Z \in \Gamma(TM_{2T})$ where $\sigma:=\sigma_{p,q}=\frac{p+\sqrt{p^{2}+4q}}{2}$ is the metallic number ($p, q \in N^{*}$) and $\overline{\sigma}:=p-\sigma$.
\end{prop}
\begin{proof}
For any $X \in \Gamma(TM_{1T})$ (i.e. $JX=NX$), $Z \in \Gamma(TM_{2T})$ (i.e. $JZ=TZ$) and $\overline{W} \in \Gamma(TM)$, from (\ref{e101}) and (\ref{e11}) we get
$$ q\overline{g}(\nabla_{Z}X,\overline{W})=\overline{g}(\overline{\nabla}_{Z}JX,J\overline{W})-p\overline{g}(\overline{\nabla}_{Z}X,J\overline{W}).$$
Thus, we have
$$ q\overline{g}(\nabla_{Z}X,\overline{W})=\overline{g}(\nabla_{Z}TX,T\overline{W})+\overline{g}(h(Z,TX),N\overline{W})-p\overline{g}(\nabla_{Z}X,J\overline{W}).$$
For $\overline{W} \in \Gamma(TM)$ we have $\overline{W} = W_1 + W_{2}$, where $W_{1} \in \Gamma(TM_{1T})$ and $W_{2} \in \Gamma(TM_{2T})$.
Thus, $J\overline{W} = JW_1 + JW_{2}=TW_{1}+TW_{2}=T\overline{W}$ and $N\overline{W} =0$.
By using Lemma 1(2), we obtain
 $$  \overline{g}(q X(\ln f) Z,\overline{W})= \overline{g}(TX(\ln f)TZ,\overline{W})-p\overline{g}(X(\ln f) Z,T\overline{W}),$$
 which implies
 $$  \overline{g}(q X(\ln f) Z,\overline{W})= \overline{g}((TX(\ln f) - pX(\ln f))JZ,\overline{W}),$$
 for any $X \in \Gamma(TM_{1T})$, $Z \in \Gamma(TM_{2T})$ and $\overline{W} \in \Gamma(TM)$, which implies
\begin{equation}\label{e104}
(TX(\ln f) - p X(\ln f))JZ = q X(\ln f) Z.
 \end{equation}
 Applying $J$ in (\ref{e104}) and using (\ref{e1}), we obtain
 \begin{equation}\label{e105}
[q X(\ln f) -p(TX(\ln f) - p X(\ln f))]JZ = q (TX(\ln f) - p X(\ln f) Z.
 \end{equation}
 Using the proportionality of the coefficients of $JZ$ and $Z$, respectively, from the equalities (\ref{e104}) and (\ref{e105}), we obtain
 \begin{equation}\label{e106}
\left(\frac{TX(\ln f)}{X(\ln f)}\right)^{2}-p \cdot\frac{TX(\ln f)}{X(\ln f)}-q=0.
 \end{equation}

Denoting by $\alpha =: \frac{TX(\ln f)}{X(\ln f)}$ in (\ref{e106}), we obtain the equation verified by the metallic number, $\alpha^{2}-p\alpha -q=0$, with the solutions $\sigma=\frac{p+\sqrt{p^{2}+4q}}{2}$, $\overline{\sigma}:=p-\sigma$ and from here we obtain (\ref{e103}).
\end{proof}

\begin{prop} \label{p4}
Let $M:={M_{1\perp}}\times_f {M_{2\perp}}$ be a warped product submanifold in a locally metallic (or locally Golden) Riemannian manifold $(\overline{M},\overline{g},J)$, where $M_{1\perp}$ and $M_{1\perp}$ are anti-invariant submanifolds in $\overline{M}$. Then, the warped function $f$ is constant on the connected components of $M_{1\perp}$ if and only if
\begin{equation}\label{e107}
t\nabla_{Z}^{\perp}NX=pth(X,Z),
\end{equation}
for any $X \in \Gamma(TM_{1\perp})$ and $Z \in \Gamma(TM_{2\perp})$, where $\nabla^{\perp}$ is the normal connection on $\Gamma(T^{\perp}M)$.
\end{prop}
\begin{proof}

For $\overline{W} \in \Gamma(TM)$ we have $\overline{W} = W_1 + W_{2}$, where $W_{1} \in \Gamma(TM_{1\perp})$ and $W_{2} \in \Gamma(TM_{2\perp})$.
Thus, $J\overline{W} = JW_1 + JW_{2}=NW_{1}+NW_{2}=N\overline{W}$ and $T\overline{W} =0$.
By using (\ref{e101}) and (\ref{e11}), we obtain
$$ q\overline{g}(\nabla_{Z}X,\overline{W})=\overline{g}(\overline{\nabla}_{Z}JX,J\overline{W})-p\overline{g}(\overline{\nabla}_{Z}X,J\overline{W}),$$
for any $X \in \Gamma(TM_{1\perp})$ (i.e. $JX=NX$), $Z \in \Gamma(TM_{2\perp})$ (i.e. $JZ=NZ$).

Thus, we get
$$ q\overline{g}(\nabla_{Z}X,\overline{W})=\overline{g}(\overline{\nabla}_{Z}NX,J\overline{W})-p\overline{g}(\overline{\nabla}_{Z}X,N\overline{W})$$
and, from here we get
$$ q\overline{g}(\nabla_{Z}X,\overline{W})=\overline{g}( \nabla_{Z}^{\perp}NX,N\overline{W})-p\overline{g}(h(X,Z),N\overline{W}).$$
Thus, by using Lemma 1(2) and (\ref{e8}), we obtain
 $$  \overline{g}(q X(\ln f) Z,\overline{W})= \overline{g}(t\nabla_{Z}^{\perp}NX-pth(X,Z),\overline{W}),$$
 for any $X \in \Gamma(TM_{1\perp})$, $Z \in \Gamma(TM_{2\perp})$ and $\overline{W} \in \Gamma(TM)$, which implies
 $$  q X(\ln f) Z= t\nabla_{Z}^{\perp}NX-pth(X,Z).$$
Therefore, the warped function $f$ is constant on the connected components of $M_{1\perp}$ if and only if (\ref{e107}) occurs.
\end{proof}

\begin{theorem}
Let $M:={M_{T}}\times_f {M_{\theta}}$ be a warped product semi-slant submanifold in a locally metallic (or locally Golden) Riemannian manifold $(\overline{M},\overline{g},J)$ (i.e. $M_{T}$ is invariant and $M_{\theta}$ is a proper slant submanifold in $\overline{M}$, with the slant angle $\theta \in (0, \frac{\pi}{2})$). Then $M:={M_{T}}\times_f {M_{\theta}}$ is a non proper warped product submanifold in $\overline{M}$ (i.e. the warping function $f$ is constant on the connected components of $M_{T}$).
\end{theorem}

\begin{proof}
For any $X \in \Gamma(TM_{T})$ (i.e. $JX=TX$) and $Z \in \Gamma(TM_{\theta})$, from $\overline{\nabla}J=0$, we get $\overline{\nabla}_{Z}JX = J\overline{\nabla}_{Z}X$ (where $\overline{\nabla}$ denotes the Levi-Civita connection on $\overline{M}$) and by using Lemma 1(2) and (\ref{e6}), we obtain:
$$
TX(\ln f)Z+h(TX,Z)=T\nabla_{Z}X+N\nabla_{Z}X+th(X,Z)+n h(X,Z)
$$
Thus, from the equality of the normal parts of the last equation, it follows
\begin{equation}\label{e32}
 h(TX,Z) = X(\ln f)NZ + n h(X,Z).
\end{equation}
By using $JX=TX$, for any $X \in \Gamma(TM_{T})$ and replacing $X$ with $TX$ in (\ref{e32}), we get
 $
 h(J^{2}X,Z)=TX(\ln f)NZ+n h(TX,Z)
 $
 and applying $\overline{g}(\cdot,NZ)$ in the last equality, we have
 $$
TX(\ln f)\overline{g}(NZ,NZ)= \overline{g}(h(J^{2}X,Z),NZ)-\overline{g}(n h(TX,Z),NZ)
 $$
 and using (\ref{e1}), we get
 $$
TX(\ln f)\overline{g}(NZ,NZ)=p \overline{g}(h(TX,Z),NZ) +q \overline{g} (h(X,Z),NZ) -\overline{g}(nh(TX,Z),NZ),
 $$
 for any $X \in \Gamma(TM_{T})$, $Z \in \Gamma(TM_{\theta})$.
 By using (\ref{e27})(ii) we have $\overline{g}(h(TX,Z),NZ)= 0$ and $\overline{g} (h(X,Z),NZ) =0$, for any $X \in \Gamma(TM_{T})$ and $Z \in \Gamma(TM_{\theta})$ and (\ref{e25}), we obtain
 \begin{equation}\label{e33}
TX(\ln f)\sin^{2}\theta [p \overline{g}(TZ,Z)+q \overline{g}(Z,Z)] = -\overline{g}(nh(TX,Z),NZ).
 \end{equation}
 Moreover, by using (\ref{e8})(ii) and (\ref{e7})(ii), we have
 $$\overline{g}(nh(TX,Z),NZ)=\overline{g}(h(TX,Z),nNZ)=p\overline{g}(h(TX,Z),NZ)-\overline{g}(h(TX,Z),NTZ)=0$$ because
 $TX \in \Gamma(TM_{T})$ and $TZ \in \Gamma(TM_{\theta})$.

 From (\ref{e27})(ii) we have $\overline{g}(h(TX,Z),NZ)=\overline{g}(h(TX,Z),NTZ)=0$.
 Thus, from (\ref{e33}) we obtain $TX(\ln f)\sin^{2}\theta [p \overline{g}(TZ,Z)+q \overline{g}(Z,Z)] =0$ and using (\ref{e23})(i) we get
 $$TX(\ln f)\sin^{2}\theta  \overline{g}\left(\frac{1}{\cos^2{\theta}}T^{2}Z,Z\right) =0.$$
 Therefore, from  (\ref{e7})(i) we obtain $TX(\ln f)\tan^{2}\theta  \overline{g}(TZ,TZ)=0$, for any $Z \in \Gamma(TM_{\theta})$. Moreover, for $\theta \neq \frac{\pi}{2}$ we get $TZ \neq  0 $ and we obtain $TX(\ln f) = 0$ for any $X \in \Gamma(TM_{T})$. Thus, the warping function $f$ is constant on $M_{T}$.
 \end{proof}

We shall construct some examples for the non trivial case of a warped product submanifold $M:={M_{\theta}}\times_f {M_{T}}$ in a metallic (and Golden) Riemannian manifold.

\begin{example}
Let $\mathbb{R}^{4}$ be the Euclidean space endowed with the usual Euclidean metric $\langle\cdot,\cdot\rangle$.
Let $i: M \rightarrow \mathbb{R}^{4}$ be the immersion given by:
$$i(f,\alpha,\beta):=(f \sin \alpha, f \cos \alpha, f \sin \beta, f \cos \beta),$$
where $M :=\{(f,\alpha,\beta) \mid  f>0, \alpha, \beta \in (0, \frac{\pi}{2})\}$ and $\sigma:=\sigma_{p,q}=\frac{p+\sqrt{p^{2}+4q}}{2}$ is the metallic number ($p, q \in N^{*}$).

We can find a local orthonormal frame on $TM$ given by:
 $$Z_{1}= \sin \alpha \frac{\partial}{\partial x_{1}} + \cos \alpha \frac{\partial}{\partial x_{2}}+ \sin \beta \frac{\partial}{\partial x_{3}} + \cos \beta  \frac{\partial}{\partial x_{4}},$$
 $$ Z_{2}=f \cos \alpha \frac{\partial}{\partial x_{1}} - f \sin \alpha \frac{\partial}{\partial x_{2}}, \quad
 Z_{3}=f \cos \beta \frac{\partial}{\partial x_{3}} - f \sin \beta \frac{\partial}{\partial x_{4}}.$$

We define the metallic structure $J : \mathbb{R}^{4} \rightarrow \mathbb{R}^{4} $ by:
$$
 J(X_{1},X_{2},X_{3},X_{4}):=(\sigma X_{1},\sigma X_{2}, \overline{\sigma} X_{3},\overline{\sigma} X_{4}),
 $$
which verifies $J^{2}X=p J + q I$ and $\langle JX, Y\rangle= \langle X, JY\rangle$, for any $X$, $Y\in \mathbb{R}^{4}$.
Since
 $$JZ_{1}=\sigma \sin \alpha \frac{\partial}{\partial x_{1}} +\sigma \cos \alpha \frac{\partial}{\partial x_{2}}+\overline{\sigma} \sin \beta \frac{\partial}{\partial x_{3}} + \overline{\sigma}\cos \beta  \frac{\partial}{\partial x_{4}}, \quad JZ_{2}= \sigma Z_{2}, \quad JZ_{3}= \overline{\sigma} Z_{3},$$
we remark that $\langle JZ_{1}, Z_{1}\rangle = \sigma + \overline{\sigma}=p$ and $J Z_{2}, JZ_{3} \subseteq span \{Z_{2},Z_{3}\}$.
We find that $\|Z_{1}\|^{2}=2$ and $\|Z_{2}\|^{2}= \|Z_{3}\|^{2}=f^{2}$.

Let us consider $D_{1}= span \{Z_{1}\}$ the slant distribution with the slant angle $\theta = \arccos \frac{\sigma + \overline{\sigma}}{\sqrt{2(\sigma^{2} + \overline{\sigma}^{2})}}$ and $D_{2}= span \{Z_{2}, Z_{3}\}$ the invariant distribution with respect to $J$. The distributions $D_{1}$ and $D_{2}$ satisfy the conditions from Definition \ref{d1} and they are completely integrable. Let $M_{\theta}$ and $M_{T}$ be the integral manifolds of $D_{1}$ and $D_{2}$, respectively.

Therefore, $M := M_{\theta} \times_{f}M_{T}$, with the Riemannian metric tensor
$$g:=2 df^{2} + f^{2} (d \alpha^{2} + d\beta^{2})=g_{M_{\theta}}+ f^{2} g_{M_{T}}$$
is a warped product semi-slant submanifold in the metallic Riemannian manifold $(\mathbb{R}^{4}, \langle\cdot,\cdot\rangle, J)$.

In particular, for $p=q=1$ and $\phi:=\sigma_{1,1}=\frac{1+\sqrt{5}}{2}$ the Golden number, the immersion $i: M \rightarrow \mathbb{R}^{4}$ is given by
$$i(f,\alpha,\beta):=(f \sin \alpha, f \cos \alpha, f \sin \beta, f \cos \beta, \sqrt{\phi} f),$$
and the Golden structure $J : \mathbb{R}^{4} \rightarrow \mathbb{R}^{4} $ is defined by
$$ J(X_{1},X_{2},X_{3},X_{4}):=(\phi X_{1},\phi X_{2}, \overline{\phi} X_{3},\overline{\phi}X_{4}),$$
where $\overline{\phi}=1-\phi$.

For $M_{\theta}$ the integral manifold of the slant distribution $D_{1}:=span\{Z_{1}\}$, with the slant angle $\theta = \arccos \frac{1}{\sqrt{6}}$ and $M_{T}$ the integral manifold of the invariant distribution $D_{2}:=span\{Z_{2}, Z_{3}\}$, then $M:=M_{\theta} \times_fM_{T}$ is a warped product semi-slant submanifold in the Golden Riemannian manifold $(\mathbb{R}^{4}, \langle\cdot,\cdot\rangle, J)$, with the metric $$g:=2 df^{2} + f^{2} (d \alpha^{2} + d\beta^{2})=g_{M_{\theta}}+ f^{2} g_{M_{T}}.$$
\end{example}

\begin{example}
Let $\mathbb{R}^{8}$ be the Euclidean space endowed with the usual Euclidean metric $\langle\cdot,\cdot\rangle$.
Let $i: M \rightarrow \mathbb{R}^{8}$ be the immersion given by:
$$i(f_{1},f_{2},\alpha,\beta):=(f_{1} \cos \alpha, f_{2} \cos \alpha, f_{1} \cos \beta, f_{2} \cos \beta, f_{1} \sin \alpha, f_{2} \sin \alpha, f_{1} \sin \beta, f_{2} \sin \beta),$$
where $M :=\{(f_{1},f_{2},\alpha,\beta) \mid   f_{1}, f_{2}>0, \alpha , \beta \in (0, \frac{\pi}{2})\}$.

We can find a local orthonormal frame on $TM$ given by:
 $$Z_{1}= \cos \alpha \frac{\partial}{\partial x_{1}} + \cos \beta \frac{\partial}{\partial x_{3}}+ \sin \alpha \frac{\partial}{\partial x_{5}} + \sin \beta \frac{\partial}{\partial x_{7}},$$
 $$Z_{2}= \cos \alpha \frac{\partial}{\partial x_{2}} + \cos \beta \frac{\partial}{\partial x_{4}}+ \sin \alpha \frac{\partial}{\partial x_{6}} + \sin \beta \frac{\partial}{\partial x_{8}},$$
 $$  Z_{3}= - f_{1} \sin \alpha \frac{\partial}{\partial x_{1}} - f_{2} \sin \alpha \frac{\partial}{\partial x_{2}}+ f_{1} \cos \alpha \frac{\partial}{\partial x_{5}} +  f_{2} \cos \alpha \frac{\partial}{\partial x_{6}},$$
$$ Z_{4}= - f_{1} \sin \beta \frac{\partial}{\partial x_{3}} - f_{2} \sin \beta \frac{\partial}{\partial x_{4}}+ f_{1} \cos \beta \frac{\partial}{\partial x_{7}} +  f_{2} \cos \beta \frac{\partial}{\partial x_{8}}.$$

We define the metallic structure $J : \mathbb{R}^{8} \rightarrow \mathbb{R}^{8} $ by:
$$ J(X_{1},X_{2},X_{3},X_{4},X_{5},X_{6},X_{7},X_{8}):=(\sigma X_{1},\sigma X_{2}, \overline{\sigma} X_{3},\overline{\sigma} X_{4},\sigma X_{5},\sigma X_{6}, \overline{\sigma} X_{7},\overline{\sigma} X_{8}),$$
which verifies $J^{2}=p J + q I$ and $\langle JX, Y\rangle = \langle X, JY\rangle$, for any $X,Y \in \mathbb{R}^{8}$, where $\sigma:=\sigma_{p,q}=\frac{p+\sqrt{p^{2}+4q}}{2}$ is the metallic number ($p, q \in N^{*}$). Since
 $$JZ_{1}=\sigma\cos \alpha \frac{\partial}{\partial x_{1}} + \overline{\sigma}\cos \beta \frac{\partial}{\partial x_{3}}+ \sigma\sin \alpha \frac{\partial}{\partial x_{5}} +\overline{\sigma} \sin \beta \frac{\partial}{\partial x_{7}},$$
 $$JZ_{2}=\sigma\cos \alpha \frac{\partial}{\partial x_{2}} + \overline{\sigma}\cos \beta \frac{\partial}{\partial x_{4}}+ \sigma\sin \alpha \frac{\partial}{\partial x_{6}} +\overline{\sigma} \sin \beta \frac{\partial}{\partial x_{8}},$$
 $$JZ_{3}= \sigma Z_{3}, \quad JZ_{4}= \overline{\sigma} Z_{4},$$
we remark that $\langle JZ_{1}, Z_{1}\rangle = \langle JZ_{2}, Z_{2}\rangle=\sigma + \overline{\sigma} =p $, $\|Z_{1}\|^{2}=\|Z_{2}\|^{2}=\sqrt{2}$.

Let us consider $D_{1}= span\{Z_{1},Z_{2}\}$ the slant distribution with the slant angle $\theta = \arccos \frac{\sigma + \overline{\sigma}}{\sqrt{2(\sigma^{2} + \overline{\sigma}^{2})}}$ and $D_{2}= span\{Z_{3}, Z_{4}\}$ an invariant distribution with respect to $J$. The distributions $D_{1}$ and $D_{2}$ satisfy the conditions from Definition \ref{d1} and they are completely integrable.
Let $M_{\theta}$ and $M_{T}$ be the integral manifolds of $D_{1}$ and $D_{2}$, respectively.

Therefore, $M := M_{\theta} \times_{\sqrt{f_{1}^{2}+f_{2}^{2}}}M_{T}$, with the Riemannian metric tensor
$$g:=2(df_{1}^{2} + df_{2}^{2}) + (f_{1}^{2}+f_{2}^{2}) (d \alpha^{2} + d\beta^{2})=g_{M_{\theta}}+ (f_{1}^{2}+f_{2}^{2}) g_{M_{T}}$$ is a warped product semi-slant submanifold in the metallic Riemannian manifold $(\mathbb{R}^{8}, \langle\cdot,\cdot\rangle, J)$.

In particular, for $p=q=1$ and $\phi:=\sigma_{1,1}=\frac{1+\sqrt{5}}{2}$ the Golden number, the immersion $i: M \rightarrow \mathbb{R}^{8}$ is given by
$$i(f_{1},f_{2},\alpha,\beta):=(f_{1} \cos \alpha, f_{2} \cos \alpha, f_{1} \cos \beta, f_{2} \cos \beta, f_{1} \sin \alpha, f_{2} \sin \alpha, f_{1} \sin \beta, f_{2} \sin \beta),$$
and the Golden structure $J : \mathbb{R}^{8} \rightarrow \mathbb{R}^{8} $ is defined by
$$J(X_{1},X_{2},X_{3},X_{4},X_{5},X_{6},X_{7},X_{8}):=(\phi X_{1},\phi X_{2}, \overline{\phi} X_{3},\overline{\phi} X_{4},\phi X_{5},\phi X_{6}, \overline{\phi} X_{7},\overline{\phi} X_{8}),$$
where $\overline{\phi}=1-\phi$.
If $M_{\theta}$ is the integral manifold of the slant distribution $D_{1}:=span\{Z_{1},Z_{2}\}$ with the slant angle $\theta = \arccos \frac{1}{\sqrt{6}}$ and $M_{T}$ is the integral manifold of the invariant distribution $D_{2}= span\{Z_{3}, Z_{4}\}$ with respect to $J$, then $M: =M_{\theta} \times_{\sqrt{f_{1}^{2}+f_{2}^{2}}}M_{T}$, with the metric $$g:=2(df_{1}^{2} + df_{2}^{2}) + (f_{1}^{2}+f_{2}^{2}) (d \alpha^{2} + d\beta^{2})=g_{M_{\theta}}+ (f_{1}^{2}+f_{2}^{2}) g_{M_{T}}$$ is a warped product semi-slant submanifold in the Golden Riemannian manifold $(\mathbb{R}^{8}, \langle\cdot,\cdot\rangle, J)$.
\end{example}

\begin{prop}
Let $M:={M_{\perp}}\times_f {M_{\theta}}$ (or $M:={M_{\theta}}\times_f {M_{\perp}}$) be a warped product hemi-slant submanifold in a locally metallic Riemannian manifold $(\overline{M},\overline{g},J)$ (i.e. $M_{\perp}$ is anti-invariant and $M_{\theta}$ is a proper slant submanifold in $\overline{M}$ with the slant angle $\theta \in (0, \frac{\pi}{2})$). Then $M$ is a non proper warped product submanifold in $\overline{M}$ (i.e. the warping function $f$ is constant on the connected components of $M_{\perp}$) if and only if
\begin{equation}\label{e108}
A_{NZ}X = A_{NX}Z,
\end{equation}
for any $X \in \Gamma(TM_{\perp})$, $Z \in \Gamma(TM_{\theta})$ (or $X \in \Gamma(TM_{\theta})$, $Z \in \Gamma(TM_{\perp})$, respectively).
\end{prop}

\begin{proof}
Let $M:={M_{\perp}}\times_f {M_{\theta}}$ be a warped product hemi-slant submanifold in a locally metallic Riemannian manifold $(\overline{M},\overline{g},J)$ and $X \in \Gamma(TM_{\perp})$ (i.e. $JX=NX$), $Z \in \Gamma(TM_{\theta})$ (i.e. $JZ=TZ+NZ$).

From $\overline{\nabla}J=0$, we have $\overline{\nabla}_{Z}JX=J\overline{\nabla}_{Z}X$ and using (\ref{e6}) and (\ref{e11}), we get:
$$ -A_{NX}Z+\nabla_{Z}^{\perp}NX = J(\nabla_{Z}X+h(Z,X)).$$
By using Lemma 1(2), we obtain
\begin{equation}\label{e109}
 -A_{NX}Z+\nabla_{Z}^{\perp}NX = X(\ln f)TZ + X(\ln f)NZ + th(Z,X)+nh(Z,X).
\end{equation}
From the equality of tangential component from (\ref{e109}), we obtain
\begin{equation}\label{e110}
 -A_{NX}Z = X(\ln f)TZ + th(Z,X),
\end{equation}
for any $X \in \Gamma(TM_{\perp})$, $Z \in \Gamma(TM_{\theta})$.
From $\overline{\nabla}_{X}JZ=J\overline{\nabla}_{X}Z$ and using (\ref{e6}) and (\ref{e11}), we get:
$$ \nabla_{X}TZ+h(X,TZ)-A_{NZ}X+\nabla_{X}^{\perp}NZ = J(\nabla_{X}Z+h(X,Z)).$$
By using Lemma 1(2), we obtain
\begin{equation}\label{e111}
X(\ln f)TZ+h(X,TZ)-A_{NX}Z+\nabla_{X}^{\perp}ZX = X(\ln f)TZ + X(\ln f)NZ + th(Z,X)+nh(Z,X).
\end{equation}
From the equality of tangential component from (\ref{e111}), we obtain
\begin{equation}\label{e112}
 -A_{NZ}X = th(Z,X),
\end{equation}
for any $X \in \Gamma(TM_{\perp})$, $Z \in \Gamma(TM_{\theta})$.
Thus, from (\ref{e110}) and (\ref{e112}) we get
\begin{equation}\label{e113}
X(\ln f)TZ = A_{NZ}X - A_{NX}Z,
\end{equation}
for any $X \in \Gamma(TM_{\perp})$, $Z \in \Gamma(TM_{\theta})$.
Therefore, the warping function $f$ is constant on the connected components of $M_{\perp}$ if and only if (\ref{e108}) holds.

For $M:={M_{\theta}}\times_f {M_{\perp}}$, we use a similar proof to obtain the conclusion.
\end{proof}

\begin{remark}
 In a metallic (or Golden) Riemannian manifold $(\overline{M},\overline{g},J)$ there exist proper hemi-slant warped product submanifolds of the form $M:={M_{\perp}}\times_f {M_{\theta}}$ or $M:={M_{\theta}}\times_f {M_{\perp}}$, where  $M_{\perp}$ is anti-invariant submanifold and $M_{\theta}$ is a proper slant submanifold in $\overline{M}$, with the slant angle $\theta \in (0, \frac{\pi}{2})$, as we can see in the next examples.
\end{remark}

\begin{example}
Let $\mathbb{R}^{5}$ be the Euclidean space endowed with the usual Euclidean metric $\langle\cdot,\cdot\rangle$.
Let $i: M \rightarrow \mathbb{R}^{5}$ be the immersion given by:
$$i(f,\alpha):=\left(f \sin \alpha, f \cos \alpha, f \sin \alpha, f \cos \alpha, \sqrt{\frac{p \sigma}{q}}f\right),$$
where $M :=\{(f,\alpha) \mid  f>0, \alpha \in (0, \frac{\pi}{2})\}$ and $\sigma:=\sigma_{p,q}=\frac{p+\sqrt{p^{2}+4q}}{2}$ is the metallic number ($p, q \in N^{*}$).

We can find a local orthonormal frame on $TM$ given by:
 $$Z_{1}= \sin \alpha \frac{\partial}{\partial x_{1}} + \cos \alpha \frac{\partial}{\partial x_{2}}+ \sin \alpha \frac{\partial}{\partial x_{3}} + \cos \alpha  \frac{\partial}{\partial x_{4}}+ \sqrt{\frac{p \sigma}{q}}\frac{\partial}{\partial x_{5}},$$
 $$ Z_{2}=f \cos \alpha \frac{\partial}{\partial x_{1}} - f \sin \alpha \frac{\partial}{\partial x_{2}} + f \cos \alpha \frac{\partial}{\partial x_{3}} - f \sin \alpha \frac{\partial}{\partial x_{4}}.$$

We define the metallic structure $J : \mathbb{R}^{5} \rightarrow \mathbb{R}^{5} $ by:
$$
 J(X_{1},X_{2},X_{3},X_{4},X_{5}):=(\sigma X_{1},\sigma X_{2}, \overline{\sigma} X_{3},\overline{\sigma} X_{4}, \overline{\sigma} X_{5} ),
 $$
which verifies $J^{2}X=p J + q I$ and $\langle JX, Y\rangle = \langle X, JY\rangle$, for any $X$, $Y \in \mathbb{R}^{5}$.
Since
 $$JZ_{1}=\sigma \sin \alpha \frac{\partial}{\partial x_{1}} +\sigma \cos \alpha \frac{\partial}{\partial x_{2}}+\overline{\sigma} \sin \alpha \frac{\partial}{\partial x_{3}} + \overline{\sigma}\cos \alpha  \frac{\partial}{\partial x_{4}}+ \overline{\sigma}\sqrt{\frac{p \sigma}{q}}\frac{\partial}{\partial x_{5}},$$
 $$ JZ_{2}= \sigma f \cos \alpha \frac{\partial}{\partial x_{1}} - \sigma f \sin \alpha \frac{\partial}{\partial x_{2}} + \overline{\sigma} f \cos \alpha \frac{\partial}{\partial x_{3}} - \overline{\sigma} f \sin \alpha \frac{\partial}{\partial x_{4}},$$
we remark that $JZ_{1} \perp span \{Z_{1},Z_{2}\} = TM$ and $\langle JZ_{2}, Z_{2}\rangle = f^{2} (\sigma + \overline{\sigma})=f^{2}p$.
We find that $\|Z_{1}\|^{2}=1+\frac{\sigma^{2}}{q}$ and $\|Z_{2}\|^{2}= 2f^{2}$.

Let us consider $D_{1}= span\{Z_{1}\}$ the anti-invariant distribution (with respect to $J$) and $D_{2}= span\{Z_{2}\}$ the slant distribution  with the slant angle $\theta = \arccos \frac{\sigma + \overline{\sigma}}{\sqrt{2(\sigma^{2} + \overline{\sigma}^{2})}}$. The distributions $D_{1}$ and $D_{2}$ satisfy the conditions from Definition \ref{d1} and they are completely integrable.
Let $M_{\perp}$ and $M_{\theta}$ be the integral manifolds of $D_{1}$ and $D_{2}$, respectively.

Therefore, $M :=M_{\perp} \times_{\sqrt{2}f} M_{\theta}$ with the Riemannian metric tensor
$$g:=\left(1+\frac{\sigma^{2}}{q}\right) df^{2} + 2f^{2} d \alpha^{2}=g_{M_{\perp}}+ 2f^{2} g_{M_{\theta}}$$
is a warped product hemi-invariant submanifold in the metallic Riemannian manifold $(\mathbb{R}^{5}, \langle\cdot,\cdot\rangle, J)$.

In particular, for $p=q=1$ and $\phi:=\sigma_{1,1}=\frac{1+\sqrt{5}}{2}$ the Golden number, the immersion $i: M \rightarrow \mathbb{R}^{5}$ is given by
$$i(f,\alpha,\beta):=(f \sin \alpha, f \cos \alpha, f \sin \beta, f \cos \beta, \sqrt{\phi} f),$$
and the Golden structure $J : \mathbb{R}^{5} \rightarrow \mathbb{R}^{5} $ is defined by
$$ J(X_{1},X_{2},X_{3},X_{4},X_{5}):=(\phi X_{1},\phi X_{2}, \overline{\phi} X_{3},\overline{\phi}X_{4}, \overline{\phi}X_{5}),$$
where $\overline{\phi}=1-\phi$.
If $M_{\perp}$ is the integral manifold of the anti-invariant distribution $D_{1}:=span\{Z_{1}\}$ and $M_{\theta}$ is the integral manifold of the slant distribution $D_{2}:=span\{Z_{2}\}$ with the slant angle $\theta = \arccos \frac{1}{\sqrt{6}}$, then $M :=M_{\perp} \times_{\sqrt{2}f} M_{\theta}$ with the metric $$g:=(1+\phi^2) df^{2} +2 f^{2} d \alpha^{2}= g_{M_{\perp}}+ 2f^{2} g_{M_{\theta}}$$ is a warped product hemi-slant submanifold in the Golden Riemannian manifold $(\mathbb{R}^{5}, \langle\cdot,\cdot\rangle, J)$.
\end{example}

\begin{example}
Let $\mathbb{R}^{7}$ be the Euclidean space endowed with the usual Euclidean metric $\langle\cdot,\cdot\rangle$.
Let $i: M \rightarrow \mathbb{R}^{7}$ be the immersion given by:
$$i(f,\alpha):=\left(f \sin \alpha, f \cos \alpha, \frac{\sigma}{\sqrt{q}} f \sin \alpha,  \frac{\sigma}{\sqrt{q}} f \cos \alpha, \frac{1}{\sqrt{2}}f,\frac{1}{\sqrt{2}}f, -f \right),$$
where $M :=\{(f,\alpha) \mid  f>0, \alpha \in (0, \frac{\pi}{2})\}$ and $\sigma:=\sigma_{p,q}=\frac{p+\sqrt{p^{2}+4q}}{2}$ is the metallic number ($p, q \in N^{*}$).

We can find a local orthonormal frame on $TM$ given by:
 $$Z_{1}= \sin \alpha \frac{\partial}{\partial x_{1}} + \cos \alpha \frac{\partial}{\partial x_{2}}+ \frac{\sigma}{\sqrt{q}} \sin \alpha \frac{\partial}{\partial x_{3}} + \frac{\sigma}{\sqrt{q}}\cos \alpha  \frac{\partial}{\partial x_{4}}+ \frac{1}{\sqrt{2}} \frac{\partial}{\partial x_{5}}+ \frac{1}{\sqrt{2}} \frac{\partial}{\partial x_{6}} - \frac{\partial}{\partial x_{7}}$$
 $$ Z_{2}=f \cos \alpha \frac{\partial}{\partial x_{1}} - f \sin \alpha \frac{\partial}{\partial x_{2}} +\frac{\sigma}{\sqrt{q}} f \cos \alpha \frac{\partial}{\partial x_{3}} -
 \frac{\sigma}{\sqrt{q}}f \sin \alpha \frac{\partial}{\partial x_{4}}.$$

We define the metallic structure $J : \mathbb{R}^{7} \rightarrow \mathbb{R}^{7} $ by:
$$
 J(X_{1},X_{2},X_{3},X_{4},X_{5},X_{6},X_{7}):=(\sigma X_{1},\sigma X_{2}, \overline{\sigma} X_{3},\overline{\sigma} X_{4},\sigma X_{5},\sigma X_{6}, \overline{\sigma} X_{7} ),
 $$
which verifies $J^{2}X=p J + q I$ and $\langle JX, Y\rangle = \langle X, JY\rangle$, for any $X$, $Y \in \mathbb{R}^{5}$.
Since
 $$JZ_{1}=\sigma \left(\sin \alpha \frac{\partial}{\partial x_{1}} + \cos \alpha \frac{\partial}{\partial x_{2}}\right)-\sqrt{q} \left(\sin \alpha \frac{\partial}{\partial x_{3}} - \cos \alpha  \frac{\partial}{\partial x_{4}}\right)+ \frac{\sigma}{\sqrt{2}} \left(\frac{\partial}{\partial x_{5}}+ \frac{\partial}{\partial x_{6}}\right)-\overline{\sigma} \frac{\partial}{\partial x_{7}},$$
 $$ JZ_{2}= \sigma \left( f \cos \alpha \frac{\partial}{\partial x_{1}} - f \sin \alpha \frac{\partial}{\partial x_{2}} \right) - \sqrt{q} \left( f \cos \alpha \frac{\partial}{\partial x_{3}} - f \sin \alpha \frac{\partial}{\partial x_{4}}\right),$$
we remark that $JZ_{2} \perp span \{Z_{1},Z_{2}\}$ and $\langle JZ_{1}, Z_{1}\rangle = \sigma + \overline{\sigma}$.

We find that $\|Z_{1}\|^{2}=3+\frac{\sigma^{2}}{q}$ and $\|Z_{2}\|^{2}= f^{2}(\frac{\sigma^{2}}{q}+1)$.

We denote by $D_{1}:=span\{Z_{1}\}$ the slant distribution with the slant angle $\theta$, where $\cos \theta = \frac{\sqrt{q}(\sigma + \overline{\sigma})}{\sqrt{(\sigma^{2} +3q)( \sigma^{2}+q+2)}}$ and by $D_{2}:=span\{Z_{2}\}$ an anti-invariant distribution (with respect to $J$). The distributions $D_{1}$ and $D_{2}$ satisfy conditions from Definition \ref{d1}.

If $M_{\theta}$ and $M_{\perp}$ are the integral manifolds of the distributions $D_{1}$ and $D_{2}$, respectively, then $M := M_{\theta} \times_{\sqrt{(\frac{\sigma^{2}}{q}+1)}f}M_{\perp}$ with the Riemannian metric tensor
$$g:=\left(\frac{\sigma^{2}}{q}+3\right) df^{2} +  f^{2}\left( \frac{\sigma^{2}}{q}+1\right) d \alpha^{2}=g_{M_{\theta}}+  f^{2}\left( \frac{\sigma^{2}}{q}+1\right) g_{M_{\perp}}$$
is a warped product hemi-invariant submanifold in the metallic Riemannian manifold $(\mathbb{R}^{7}, \langle\cdot,\cdot\rangle, J)$.

In particular, for $p=q=1$ and $\phi:=\sigma_{1,1}$ is the Golden number, the immersion $i: M \rightarrow \mathbb{R}^{7}$ is given by
$$i(f,\alpha):=\left(f \sin \alpha, f \cos \alpha, \phi f \sin \alpha,  \phi f \cos \alpha, \frac{1}{\sqrt{2}}f,\frac{1}{\sqrt{2}}f, -f \right),$$
and the Golden structure $J : \mathbb{R}^{7} \rightarrow \mathbb{R}^{7} $ is defined by
$$ J(X_{1},X_{2},X_{3},X_{4},X_{5}):=(\phi X_{1},\phi X_{2}, \overline{\phi} X_{3},\overline{\phi}X_{4}, \phi X_{5}, \phi X_{6}, \overline{\phi}X_{7}),$$
where $\overline{\phi}=1-\phi$.

If $M_{\theta}$ is the integral manifold of the slant distribution $D_{1}:=span\{Z_{1}\}$ with the slant angle $\theta = \arccos \frac{1}{\phi^{2} +3}$  and $M_{\perp}$ is the integral manifold of the anti-invariant distribution $D_{2}:=span\{Z_{2}\}$, then $M := M_{\theta} \times_{\sqrt{\phi^{2}+1}f}M_{\perp}$ with the metric
$$g:= (\phi^{2}+3)df^{2} +  f^{2}(\phi^{2}+1)d \alpha^{2}=g_{M_{\theta}}+  f^{2}(\phi^{2}+1)g_{M_{\perp}}$$
is a warped product hemi-slant submanifold in the Golden Riemannian manifold $(\mathbb{R}^{7}, \langle\cdot,\cdot\rangle, J)$.

\end{example}

\linespread{1}
Cristina E. Hretcanu, \\ Stefan cel Mare University of Suceava, Romania, e-mail: criselenab@yahoo.com\\
Adara M. Blaga, \\ West University of Timisoara, Romania, e-mail: adarablaga@yahoo.com.

\end{document}